\documentclass[headings=normal]{scrartcl}
\usepackage{amsmath, amssymb, amsthm}
\usepackage[numbers]{natbib}
\usepackage[shortlabels]{enumitem}
\usepackage{cleveref}

\def\Z{\mathbb{Z}}

\def\N{\mathbb{N}}

\def\tleq{\trianglelefteq}

\setlist{itemsep=0pt,topsep=\parsep}
\setkomafont{dedication}{\normalfont }

\theoremstyle{plain}
\newtheorem{theorem}{Theorem}[section]
\newtheorem{lemma}[theorem]{Lemma}

\theoremstyle{definition}
\newtheorem{definition}[theorem]{Definition}
\newtheorem{remark}[theorem]{Remark}
\newtheorem{example}[theorem]{Example}

\newcommand\extrafootnote[1]{
    \begingroup
    \renewcommand\thefootnote{}\footnote{#1}
    \addtocounter{footnote}{-1}
    \endgroup
}

\newcommand{\NZ}{\ensuremath{\langle N, Z(G) \rangle}}
\newcommand{\GAN}{\ensuremath{G = A N}}

\newcommand{\prog}[1]{{\sffamily #1}}
\begin{document}
\title{The commuting graphs of certain cyclic-by-abelian groups}
\author{Timo Velten}
\date{November 2024}
\dedication{In memory of Michael Herrmann.}
\maketitle

\begin{abstract}
Let \(G\) be a finite, non-abelian group of the form \(\GAN\),
where \(A \leq G\) is abelian, and \(N \tleq G\) is cyclic. We prove that
the commuting graph \(\Gamma(G)\) of \(G\) is either a connected graph of diameter at most four,
or the disjoint union of \(|G'| + 1\) complete graphs. These results apply to all finite metacyclic groups,
and to groups of square-free order in particular.
\end{abstract}
\extrafootnote{I am grateful to Bettina Eick for introducing me to this topic during my Bachelor's thesis
and her support since then. I thank E. A. O'Brien for many helpful comments.}

\section{Introduction}

The commuting graph \(\Gamma(G)\) of a non-abelian finite group \(G\) is the simple
graph with vertex set \(V = \{g \in G \mid g \notin Z(G)\}\)
and vertices \(u, v \in V\) are adjacent if and only
if \(u\) and \(v\) commute in \(G\).
\citet{Iranmanesh2008} proved that the commuting graph
of a finite symmetric or alternating group is either
disconnected, or has diameter at most five. 
They conjectured that the commuting graph of any finite non-abelian group is either disconnected,
or has diameter bounded above by a universal constant independent of the group.
This conjecture was disproved by \citet{GiudiciNoUpperBound}.

Nevertheless, the conjecture of \cite{Iranmanesh2008} is true for some restricted
classes of groups.
\citet{Parker2013} proved that the commuting graph
of a solvable group with trivial center is either disconnected or has diameter at most eight.
\citet{MorganParker} proved that if a group has trivial center, then every connected
component of its commuting graph has diameter at most 10.
The commuting graph of \(\mathrm{GL}_n(\Z / m \Z)\)
is either disconnected or has diameter at most eight \cite{GiudiciIntegers}.
Recently, Carleton and Lewis \cite{Lewis2024} considered
groups with abelian Sylow subgroups;
in particular, they prove that if such a group has derived length 2
and its commuting graph is connected, then the graph has diameter at most 4.

Here we consider the following groups.

\begin{definition}\label{def:ca_factorization}
A finite group \(G\) has a
cyclic-by-abelian factorization
if there exist a normal, cyclic subgroup \(N \tleq G\) and an abelian subgroup
\(A \leq G\) such that \(\GAN\).
\end{definition}

This definition generalizes metacyclic groups. If \(G\) is metacyclic with \(G / N = \langle a N \rangle\)
and \(N\) both cyclic, then \(G\) has the cyclic-by-abelian factorization \(G = A N\), where \(A := \langle a \rangle\).

It is well-known that all groups of square-free order are metacyclic \cite[Chapter V, \mbox{§3}, Theorem 11]{zbMATH01445102}.
Therefore, the class of groups that have a cyclic-by-abelian factorization includes
all groups of square-free order.

We prove the following theorem on the structure of the commuting graphs
of groups that have a cyclic-by-abelian factorization. It extends
the results of \cite{Velten2021}.

\begin{theorem}\label[theorem]{thm:main_thm}
Let \(G\) be a finite, non-abelian group with cyclic-by-abelian factorization \(\GAN\).
Let \(\Gamma\) be its commuting graph. Then the following statements hold.
\begin{enumerate}[(a)]
    \item\label{thm:main:connected_iff} The graph \(\Gamma\) is connected if and only if there exists a non-central \(g \in G\) such that \(C_{G}(g)\) is non-abelian.
    \item\label{thm:main:connected} If \(\Gamma\) is connected, then its diameter is at most four.
    \item\label{thm:main:disconnected} \begin{sloppypar}
        If \(\Gamma\) is disconnected, then it is the disjoint union of one complete graph
        on \mbox{\((|G'| - 1) |Z(G)|\)} and \(|G'|\) complete graphs,
        each on \mbox{\([G : G'] - |Z(G)|\)} vertices.
    \end{sloppypar}
\end{enumerate}
\end{theorem}

We demonstrate that both \ref{thm:main:connected} and \ref{thm:main:disconnected} occur.

\begin{example}
Consider the dihedral group  of order \(4n\):
\[D_{4n} = \langle a, x \mid  a^2 = x^{2 n} = 1, x^a = x^{-1} \rangle.\]
It has the cyclic-by-abelian factorization
\(A \ltimes N\) with \(A = \langle a \rangle\) and \(N = \langle x \rangle\). Also,
\(|Z(D_{4n})| = 2\) and \(|D_{4n}'| = n\). Since \(C_{D_{4n}}(u) = N\) for every
non-central \(u \in N\), \Cref{lemma:equivalencies} and the theorem imply that the commuting graph of \(D_{4n}\)
is the disjoint union of one complete graph on \(2 (n - 1)\) vertices and \(n\) complete graphs
on \(2\) vertices.
\end{example}

\begin{example}
Consider the group
\begin{equation}\label{eq:group_with_con}
    G = \langle a, x \mid a^m = x^n = 1, x^a = x^2 \rangle
\end{equation}
where \(n \geq 9\) is an odd multiple of \(3\) and \(m \in \N\) such that \(2^m \equiv 1 \pmod{n}\).
Then \(G = A N\) where \(A = \langle a \rangle\) and \(N = \langle x \rangle\).
It holds that \(N \cap Z(G) = \{1\}\). Moreover, if \(s = n / 3\), then
\((x^s)^{a^2} = x^{4 s} = x^{s}\) and \(C_{G}(x^s) \not\leq \NZ\).
By \Cref{lemma:equivalencies} and the theorem, the commuting graph of \(G\)
is connected and has diameter at most four.
\end{example}

We used the computer algebra system \prog{GAP} \cite{GAP4},
the \textsc{Small Groups} library \cite{SmallGrp} and the \prog{GAP} package \prog{GRAPE} \cite{GRAPE}
to demonstrate that our results are sharp.

Various metacyclic groups have a commuting graph with diameter four, for example
\(\mathrm{SmallGroup}(60, 7)\), which is the group
defined in \eqref{eq:group_with_con} for \((n, m) = (15, 4)\). Also, \(\mathrm{SmallGroup}(96, 118)\)
has a cyclic-by-abelian factorization, is \textit{not} metacyclic and its commuting graph has diameter four.

If \(G\) has a cyclic-by-abelian factorization, then
\(G'\) is cyclic. However, this condition alone does not imply \Cref{thm:main_thm}.
For example, \(\mathrm{SmallGroup}(1040, 93)\) has a cyclic commutator
subgroup but its commuting graph has diameter five. Also, \(\mathrm{SmallGroup}(48, 15)\)
has a cyclic commutator subgroup and a disconnected commuting graph which
is not the disjoint union of complete graphs.

Following \cite{PerfectCommutingGraphs}, an AC-group is a group in which the centralizer
of every non-central element is abelian. \Cref{thm:main_thm}
asserts that a group that admits a cyclic-by-abelian factorization has a disconnected commuting
graph if and only if it is an AC-group. These were studied by \citet{Schmidt}; he refers to them as \(\mathfrak{M}\)-groups.

\section{Preliminary results and notation}

For
\(g, h \in G\) we write \(g^h := h^{-1} g h\) for conjugation and
\([g, h] := g^{-1} h^{-1} g h = g^{-1} g^h\) for their commutator. The centralizer
of \(g \in G\) is denoted by \(C_{G}(g)\), and the center of \(G\) by \(Z(G)\).

We will use the following basic result frequently.

\begin{remark}
Let \(G\) have an abelian normal subgroup \(N\). Then
\([u, g]^l = [u^l, g]\) for every \(u \in N\), \(g \in G\) and \(l \in \Z\).
\end{remark}
\begin{proof}
\([u, g]^l = (u^{-1} u^g)^l = u^{-l} (u^l)^g = [u^l, g]\).
\end{proof}

The following lemma collects properties of groups that have a cyclic-by-abelian factorization.

\begin{lemma}\label[lemma]{basic_facts}
Let \(G\) be a finite group with cyclic-by-abelian factorization \(\GAN\)
and \(N = \langle x \rangle\).
\begin{enumerate}[(i)]
\item\label{basic:commute_condition} Let \(u, v \in N\)
    and \(g, h \in G\) such that \(g\) and \(h\) commute. Then \(g u\) and \(h v\)
    commute if and only if \([u, h] = [v, g]\).
\item\label{basic:intersection} \(N \cap Z(G) = \langle x^{n} \rangle\) where \(n := |G'|\). In particular, \(|\NZ| = |G'| |Z(G)|\).
\item\label{basic:generator_case} If \([x, g]\) generates \(G'\), then \(|C_{G}(g)| = [G : G']\) and \(C_G(g)\) is abelian.
\end{enumerate}
\end{lemma}
\begin{proof}
\ref{basic:commute_condition} Note that
\[(g u) \cdot (h v) = (g h) (u^h v)
\quad \text{ and } \quad
(h v) \cdot (g u) = (h g) (v^g u).
\]
As \(g h = h g\), this implies that \(g u\) and \(h v\) commute
if and only if \(u^h v = v^g u\). Since \(N\) is abelian,
\([u, h] = u^{-1} u^h = v^{-1} v^g = [v, g]\).

\ref{basic:intersection} Let \(g \in G\). Then \(1 = [x, g]^n = [x^n, g]\), so \(x^n \in Z(G) \cap N\).
To show the other inclusion, suppose that \(x^s \in N \cap Z(G)\). Note that \(G'\)
is the normal closure
of commutators of the form \([x, a]\) for \(a \in A\), so
\(G' = \langle [x, a] \mid a \in A \rangle^G\).
For every \(a \in A\), it holds that \([x, a]^s = [x^s, a] = 1\). Therefore,
\(G'\) is generated by elements of order dividing \(s\). Moreover, \(G / N\) is abelian, so \(G' \leq N\)
is cyclic. From this, we conclude that \(n = |G'|\) must also divide \(s\), so \(x^s \in \langle x^n \rangle\).

\ref{basic:generator_case} Let \(B := \langle A, Z(G) \rangle\) and let \(T := \{1, x, \dots, x^{n - 1}\}\).
Then \(T\) is a transversal of the right cosets of \(N \cap Z(G) = B \cap N\) in \(G\),
and therefore also a transversal of the right cosets of \(B\) in \(G = B N\). In particular, every \(g \in G\)
can be uniquely written as \(g = b x^k\) with \(b \in B\) and \(0 \leq k < n\) and \(|G| = |B| |G'|\).

Now we determine \(C_{G}(g)\). Write \(g = a x^k\) with \(a \in B\) and \(0 \leq k < n\). Fix \(b \in B\).
We determine all integers \(\ell = \ell(b)\) such that \(b x^\ell\) commutes with \(g\). Because \(b\) and \(a\) commute,
we know from \Cref{basic_facts} that \(g\) and \(b x^\ell\) commute if and only if
\[[x^k, b] = [x^\ell, a] = [x, a]^{\ell} = [x, g]^{\ell}.\]
Since \([x, g]\) generates \(G'\), there is exactly one integer \(\ell\) that satisfies this equation.

This establishes the existence of a mapping \(\ell \colon B \to \{0, 1, \dots, n - 1\}\) such that
\[C_{G}(g) = \{b x^{\ell(b)} \mid b \in B\}.\]
These elements are distinct, so \(|C_{G}(g)| = |B|\). As \(|G| = |B| |G'|\),
it follows that \(|C_{G}(g)| = [ G: G']\).

We now prove that \(C_{G}(g)\) is abelian. Consider two elements \(b x^{\ell(b)}\)
and \(c x^{\ell(c)}\) of \(C_{G}(g)\).
It suffices to prove that \([x, c]^{\ell(b)} = [x, b]^{\ell(c)}\). As
both commute with \(g\)
\[[x, a]^{\ell(b)} = [x^k, b] \quad \text{ and } \quad [x, a]^{\ell(c)} = [x^k, c].\]
Furthermore, \([x, a] = [x, g]\) generates \(G'\), so there exist integers \(r_b\) and \(r_c\)
such that \([x, a]^{r_b} = [x, b]\) and \([x, a]^{r_c} = [x, c]\). We deduce that
\[[x, b]^{\ell(c)} = [x, a]^{r_b \ell(c)} = [x^k, c]^{r_b} = [x, c]^{k r_b} = [x, a]^{k r_b r_c}.\]
Similarly, \([x, c]^{\ell(b)} = [x, a]^{k r_b r_c}\) so
\([x, c]^{\ell(b)} = [x, b]^{\ell(c)}\). We conclude that \(b x^{\ell(b)}\) and \(c x^{\ell(c)}\) commute
and that \(C_{G}(g)\) is abelian.
\end{proof}

\begin{lemma}\label[lemma]{lemma:equivalencies}
Let \(G\) be a finite, non-abelian group with cyclic-by-abelian factorization \(\GAN\).
Then the following statements are equivalent.
\begin{enumerate}[(i)]
\item\label{equiv:N} \(C_{G}(u) \leq \NZ\) for every non-central \(u \in N\).
\item\label{equiv:abelian} \(C_{G}(g)\) is abelian for every non-central \(g \in G\).
\end{enumerate}
Moreover, if either of these statements hold, then for every non-central \(g \in G\)
either \(C_{G}(g) = \NZ\), or \(|C_{G}(g)| = [G : G']\),
or both hold.
\end{lemma}
\begin{proof}
\(\neg\ref{equiv:N} \implies \neg\ref{equiv:abelian}\): Let \(h \in C_{G}(x) \setminus \NZ\) with
\(x \in N\) non-central. Write \(h = a y\) with \(a \in A\) and \(y \in N\). Then \(a\) commutes with \(x\). Also, \(a\)
does not commute with every element of \(N\). If it did, then \(a\) would be central and
\(h\) would be in \(\NZ\). Therefore,
\(\langle a, N \rangle\) is a non-abelian subgroup of \(C_{G}(x)\) and \(C_{G}(x)\) is non-abelian.

\ref{equiv:N} \(\implies\) \ref{equiv:abelian}: Let \(g\) be non-central in \(G\). If
\([x, g]\) generates \(G'\), then \(C_{G}(g)\) is abelian and \(|C_{G}(g)| = [G : G']\) by \Cref{basic_facts}.

Now, suppose that \([x, g]\) does not generate \(G'\) and has order \(s < |G'|\).
Then \([x^s, g] = [x, g]^s = 1\) so \(x^s\) and \(g\) commute. Also, \(x^s\) is not central by \Cref{basic_facts}.
Therefore, \(g \in C_{G}(x^s) = \NZ\). Write \(g = y z\) with \(y \in N\) and \(z \in Z(G)\). Then
\(C_{G}(g) = C_{G}(y) = \NZ\). In particular, \(C_{G}(g)\) is abelian.
\end{proof}

\section{Proof of \Cref{thm:main_thm}}

\begin{proof}[\unskip\nopunct]
Let \(N = \langle x \rangle\) and \(n := |G'|\). We distinguish two cases.

\medskip\textit{Case \(1\): There exists a non-central \(g \in G\) such that \(C_{G}(g)\) is non-abelian}.
Then there exists \(u^* \in N\) such that \(C_{G}(u^*)\) is strictly larger than \(\NZ\) by \Cref{lemma:equivalencies}. We show that
the vertex \(u^*\) has distance at most two to every other vertex of \(\Gamma\).
That is, we prove that
\[C_{G}(g) \cap C_{G}(u^*) \neq Z(G) \quad \text{ for all } \quad g \in G.\]

If \([x, g]\) does not generate \(G'\), and has order \(s < n\), then
\[1 = [x, g]^s = [x^s, g]\]
so \(x^s\) and \(g\) commute. By \Cref{basic_facts}, \(x^s\) is not central, so \(x^s\)
is a non-central element that commutes with both \(g\) and \(u^*\).

Now suppose that \([x, g]\) generates \(G'\). By assumption, \(C_{G}(u^*)\)
is strictly larger than \(\NZ\). It follows from \Cref{basic_facts} that
\(|C_G(u*)| > |G'| |Z(G)|\) and \(|C_G(g)| = [G : G']\). Using this, we see that
\begin{align*}
    |C_{G}(g) \cap C_{G}(u^*)| = \frac{|C_{G}(g)| [G : G']}{|C_{G}(u^*) C_{G}(g)|}
    \geq \frac{C_{G}(u*)}{|G'|}
    > \frac{|\NZ|}{|G'|} = |Z(G)|
\end{align*}
so \(C_{G}(u^*) \cap C_{G}(g)\) contains at least one non-central element of \(G\).

This proves that every vertex of \(\Gamma\) has distance at most two to \(u^*\), so \(\Gamma\)
is connected and has diameter at most four.

\medskip\textit{Case \(2\): The centralizer of every non-central element of \(G\) is abelian}.
In this case, adjacency of two
vertices of \(\Gamma\) is a transitive relation. In other words, \(\Gamma\)
is a disjoint union of complete graphs and, in particular, it is disconnected.

Let \(g \in G\) be non-central. If \(g \in \NZ\), then \(C_{G}(g) = \NZ\). Therefore,
the connected component of \(g\) in \(\Gamma\) is a complete graph on \(|\NZ| - |Z(G)| =
|Z(G)|(|G'| - 1)\) vertices.
If \(g \notin \NZ\), then \Cref{lemma:equivalencies} asserts that \(|C_G(g)| = [G : G']\), so
the connected component of \(g\) has \([G : G'] - |Z(G)|\) vertices.

We conclude that \(\Gamma\) is the disjoint union of one complete graph on \(|Z(G)| (|G'| - 1)\)
vertices, and \(r\) complete graphs on \([G : G'] - |Z(G)|\), where
\[r = \frac{|\Gamma| - |Z(G)|(|G'| - 1)}{[G : G'] - |Z(G)|} = \frac{|G| - |Z(G)| |G'|}{\frac{|G|}{|G'|} - |Z(G)|} = |G'|.\qedhere\]
\end{proof}

\bibliographystyle{unsrtnat}

\end{document}